\documentclass{amsart}
\usepackage{amscd,amssymb,latexsym,verbatim}
\usepackage{hyperref}

\theoremstyle{plain}
\newtheorem{thm}{Theorem}[section]
\newtheorem{thm0}{Theorem}
\newtheorem{lem}[thm]{Lemma}
\newtheorem{prop}[thm]{Proposition} 
\newtheorem{cor}[thm]{Corollary}

\theoremstyle{definition}
\newtheorem{rem}[thm]{Remark}
\newtheorem{definition}[thm]{Definition} 
\newtheorem{example}[thm]{Example}
\newtheorem*{problem}{Problem}
\newtheorem{nothing}[thm]{\noindent\!\!\bf}

\newcommand{\A}{\mathcal{A}}
\newcommand{\EE}{\mathcal{E}} 
\newcommand{\GG}{\mathcal{G}}
\newcommand{\VV}{\mathcal{V}}
\newcommand{\B}{\mathbb{B}}
\newcommand{\C}{\mathbb{C}}
\newcommand{\F}{\mathbb{F}}
\newcommand{\Z}{\mathbb{Z}}
\newcommand{\K}{\mathbb{K}}
\renewcommand{\SS}{\mathbb{S}}
\newcommand{\T}{\mathbb{T}}
\newcommand{\g}{{\gamma}}
\newcommand{\G}{\Gamma}
\newcommand{\f}{\phi}
\newcommand{\s}{\sigma}
\renewcommand{\S}{\Sigma}

\renewcommand{\a}{{\alpha }}
\renewcommand{\b}{{\beta }}
\renewcommand{\l}{\lambda}
\renewcommand{\L}{\Lambda}

\numberwithin{equation}{section}

\title{Arrangements of hypersurfaces and Bestvina-Brady groups}

\author[E. Artal]{Enrique Artal Bartolo}

\author[J.I. Cogolludo]{Jos\'e Ignacio Cogolludo-Agust{\'i}n}
\address{Departamento de Matem\'aticas, IUMA\\
Universidad de Zaragoza\\
c/ Pedro Cerbuna 12\\
E-50009 Zaragoza SPAIN}
\email{artal@unizar.es,jicogo@unizar.es}

\author[D. Matei]{Daniel Matei}
\address{Institute of Mathematics of the Romanian Academy\\ 
P.O. Box 1-764\\ RO-014700 Bucharest ROMANIA} 
\email{Daniel.Matei@imar.ro}
\urladdr{http://www.imar.ro/\~{}dmatei}

\thanks{
Partially supported by
MTM2010-21740-C02-02. The third author is also partially supported
by grant CNCSIS PNII-IDEI 1188/2008 and FMI 53/10 (Gobierno de Arag{\'o}n).}


\begin{document}

\begin{abstract}
We show that quasi-projective Bestvina-Brady groups are fundamental groups 
of complements to hyperplane arrangements. 
Furthermore we relate other normal subgroups of right-angled Artin groups
to complements to arrangements of hypersurfaces.
We thus obtain examples of hypersurface complements whose 
fundamental groups satisfy various finiteness properties.
\end{abstract}

\maketitle

\section*{Introduction}

An important open question in the topology of complex algebraic varieties 
is J.-P. Serre's problem to characterize fundamental groups of smooth
algebraic varieties (cf.~\cite{Se}): Which finitely presented groups are quasi-projective, 
i.e. isomorphic to the fundamental group of a smooth connected quasi-projective 
variety (the complement of a normal crossing divisor in a smooth, connected, 
complex projective variety)? 
One could also consider an analogous question for quasi-K\"ahler groups, that is,
for fundamental groups of hypersurface complements in a compact K\"ahler manifold.
On the other hand, one could explore particular versions of Serre's question 
by restricting either the ambient projective variety or the divisor, or both. 
For example: Which groups can be realized as fundamental groups of hypersurface
(in particular hyperplane) arrangement complements in a projective space? or
Which finitely presented groups are projective, i.e. isomorphic to
fundamental groups of smooth connected projective varieties?

In this paper we are interested in quasi-projective groups from the point of
view of their finiteness properties. It is well known that such groups
are finitely presented, but much less is known about their higher dimensional 
finiteness properties.
In~\cite{ctcw:65}, C.T.C.~Wall introduced
general finiteness properties of groups and CW-complexes.
A group $G$ is said to be of type $\text{F}_n$ (resp. $\text{F}$) if it has an Eilenberg-MacLane 
complex $K(G,1)$ with finite $n$-skeleton (resp. finite).  
Clearly $G$ is finitely generated if and only if it is $\text{F}_1$ and
finitely presented if and only if it is $\text{F}_2$. An interesting example
of a finitely generated group which is not finitely presented 
was given by Stallings in~\cite{St:63}. While Wall's property $\text{F}_n$ has
a substantial geometric meaning, it is not easily detectable by homological algebraic
methods. 

In~\cite{Bi:81}, Bieri introduced homological counterparts of Wall's finiteness 
conditions. A group $G$ is said to be of type $\text{FP}_n$ if the trivial $\Z G$-module $\Z$ admits 
a projective resolution which is finitely generated in $\text{dimensions}\leq n$. 
The condition $\text{FP}_n$ is clearly weaker than $\text{F}_n$.
Note also that $G$ is of type $\text{FP}_1$ if and only if it is finitely generated, 
and that $G$ is of type $\text{FP}_2$ if it is finitely presented. 
But, as shown by Bestvina and Brady~\cite{BB:97} $\text{FP}_2$ does not imply 
finite presentation. Nevertheless, if $G$ is $\text{F}_2$ then $G$ is $\text{FP}_n$
if and only if it is $\text{F}_n$.

The first example of a group which is finitely presented but not of type
$\text{FP}_3$ was given by Stallings in~\cite{St:63}. 
Afterwards Bieri~\cite{Bi:76,Bi:81} generalized  Stallings' example to an 
infinite family. In a nutshell, let $\F_2 \times \cdots \times \F_2$ 
be the direct product of $r+1$ free groups, each of rank $2$. 
Then the kernel $N$ of the map taking each generator to 
$1 \in \Z$ is of type $\text{FP}_{r}$ but not $\text{FP}_{r+1}$. 
Stallings considered the cases $r\le 2$.
In~\cite{BB:97}, Bestvina and Brady generalized the Bieri-Stallings family
by presenting a systematic way of constructing groups $N$ of type $\text{F}_{r}$, 
but not of type $\text{F}_{r+1}$. 
This construction was further extended in~\cite{MR1610579,bxgz:99},
and revisited in~\cite{ps:09}. The input is a finite graph $\G$, together with
a homomorphism $\chi\to\mathbb{Z}$ of its associated right-angled Artin group $A_{\G}$.
The generalized Bestvina-Brady group $N_{\G}^{\chi}$, or Artin kernel, is then 
defined as the kernel of the homomorphism $\chi: A_{\G}\to\Z$, see~\cite{ps:09}.
The Bestvina-Brady groups from~\cite{BB:97} are recovered as $N_{\G}:=N_{\G}^{\chi}$
for the diagonal homomorphism $\chi: A_{\G}\to\Z$ sending each generator to~$1$.

The class of Bestvina-Brady groups $N_{\G}$ which are quasi-projective was determined by 
Dimca, Papadima and Suciu in~\cite{dps:08a} as those corresponding to either trees or
special multipartite complete graphs $\G$,
i.e. $N_\G=\F_{n_0}\times \ldots \times \F_{n_r}$ with either $r\geq 2$ or $n_i=1$ for some $i$. 
In~\cite{dps:08a}, it was proved in particular that the Stallings-Bieri groups are quasi-projective. 
It was first noticed by the third named author and Suciu in 2004, and pointed out in~\cite{ps:06a},
that the Stallings group may be realized as the fundamental group of the 
complement of a complex line arrangement in $\P^2$. This arrangement (together with a presentation of 
its fundamental group) was first considered by Arvola in an unpublished work from 1992 (see~\cite{Ra:97}).
An explicit isomorphism between the latter and the Stallings group can be found in~\cite{SUCIU-Betti}.
In~\cite{M07} and~\cite{ACM-artin} we reported on our generalization of this observation.
More precisely, we exhibited in~\cite{M07}, for each quasi-projective Bestvina-Brady group $N_{\G}$,
a complex line arrangement in $\P^2$ having $N_{\G}$ as the fundamental group
of the complement, and in~\cite{ACM-artin} higher dimensional 
hyperplane arrangements realizing just the Stallings-Bieri groups.
In this work we will present the geometric construction that was announced in~\cite{M07,ACM-artin}
(see Theorem~\ref{thm-111} for more details).

\begin{thm0}\label{thm1-intro}
Let $\G$ be a graph on $v$ vertices, and suppose $N_\G$ is a quasi-projective Bestvina-Brady group.
Then $N_\G$ is the fundamental group of the complement of a hyperplane arrangement $\A_\G$ in $\P^r$,
for some~$r< v$ depending only on~$\G$.
\end{thm0}

Theorem~\ref{thm1-intro} answers negatively~\cite[Question~2.10]{B:05} from Bestvina's 
problem list. Indeed, the fundamental group $G=N_\G$ of the complement of 
$\A_\G$ cannot have a finite $K(G,1)$, as long as $H_{r+1}(N_\G)$ has infinite rank.    
Moreover, as we shall see later in Corollary~\ref{cor-111}, the arrangement $\A_\G$ may be defined by real equations.
By taking a generic plane section of $\P^r$, one can realize $N_\G$ as
the fundamental group of the complement to a real line arrangement in $\P^2$.
For another approach to this result we refer to the work of Cohen, Falk and Randell~\cite{cfr:10}. 

Our geometric realization of the quasi-projective Bestvina-Brady group $N_\G$
as the fundamental group of an arrangement complement $M$, gives more than just
the fundamental group of $K=K(N_\G,1)$. In fact $M$ has the homotopy type of 
the $r$-skeleton of $K$.

\begin{thm0}\label{thmht-intro}
Let $\A_\G$ be a hyperplane arrangement in $\P^r$ with complement $M=M_\G$ such that 
$N_\G=\pi_1(M)$ is a quasi-projective Bestvina-Brady group.
Then $M$ has the homotopy type of the $r$-skeleton of $K(N_\G,1)$.
Suppose $N_\G$ is not the direct product of free groups. 
Then we have that
\begin{enumerate}
\item the homology groups $H_i(M)$ and $H_i(N_\G)$ are isomorphic
for $i\le r$.
\item the homotopy groups $\pi_i(M)$ vanish in the range $1<i<r$.
\item the homotopy group $\pi_r(M)$ does not vanish, in fact its group of coinvariants
$\pi_r^*(M)$ under the $\pi_1$-action is isomorphic to $H_{r+1}(N_\G)$. 
\end{enumerate}
\end{thm0}

One could compare this with the results in~\cite{DPS3} where examples are constructed 
of smooth projective manifolds $M$ of dimension $r\ge 2$ with $\pi_1(M)$ having the finiteness 
properties of the Bestvina-Brady groups, and $\pi_r(M)$ the next non-vanishing higher homotopy group.

The results in Theorems~\ref{thm1-intro} and~\ref{thmht-intro} extend to quasi-projective
Artin kernels beyond the Bestvina-Brady groups. We will show that an Artin kernel 
$N_{\G}^{\chi}$ is quasi-projective if $\G$ is a multipartite complete graph with $N_\G$ quasi-projective, 
and the epimorphism $\chi$ of the Artin group 
$A_\G=\F_{n_0}\times\dots\times\F_{n_r}$ sends each free factor $\F_{n_i}$ to a non-zero integer. 
More precisely, $N_{\G}^{\chi}$ will be realized as the fundamental
group of an arrangement of hypersurfaces in an $r$-dimensional weighted complex projective space 
whose weights are given by the images under $\chi$ of the generators of $A_\G$. 
In practice though we will work not in the weighted projective space, but in its associated  
torus. Thus all Artin kernels $N_{\G}^{\chi}$ realized by our construction turn out to be
fundamental groups of  complements to arrangements of hypertori in an ambient complex algebraic torus. 
The basic properties of such toric arrangements are similar to those of 
hyperplane arrangements, see monographs~\cite{dcp:11} and~\cite{orte:92}.

The description of the quasi-projective Bestvina-Brady groups provided by Theorems~\ref{thm1-intro} 
and~\ref{thmht-intro} yields a large class of hyperplane arrangements that exhibit new 
topological features. More precisely, the affine arrangements in $\C^{r+1}$ induced by
the projective arrangements $\A_\G$, cannot be hypersolvable since their groups
are not of type $\text{F}$. Thus the techniques described in~\cite{ps:02} can no longer be used to determine
the homotopy and homology groups of these arrangement complements. 
Compare also with the earlier results of~\cite{Ra:97}, obtained using the Lefschetz hyperplane theorem.
This suggests the following open problem.

\begin{problem}
Characterize the quasiprojective Artin kernels.
\end{problem}

In the following we will explain the motivation and the basic steps for the geometric construction
that will produce quasi-projective Artin kernels. Suppose $X$ is a smooth (quasi)-projective variety. 
Then any finite covering of $X$ is again a (quasi)-projective manifold (see for instance
Griffiths and Harris~\cite[p.~192]{GH} for the projective case and 
Namba~\cite[Lemma~2.1.2]{namba} for the general case using the Grauert-Remmert 
extension theorem of unramified coverings of complex spaces). 

In general, the (quasi)-projectivity property is lost when passing to infinite coverings. 
Suppose that $\pi_1(X)$ admits an epimorphism $\nu:\pi_1(X)\to\Z$. Let $Y$ be the regular connected 
infinite cyclic covering of $X$ associated with~$\nu$. 
We are interested in deciding whether or not~$Y$ is a (quasi)-projective manifold. 
The first obstruction would be the finite presentability of~$\pi_1(Y)$.
We may wonder what other properties of $X$ or $\pi_1(X)$ survive after 
infinite cyclic coverings. For example the $K(\pi,1)$ 
property for $X$, or other finiteness properties of $\pi_1(X)$ such as 
$\text{F}_n$, $\text{F}$, or~$\text{FP}_n$.  

In this context, the general problem that motivates our construction is the following:

\begin{problem}
Identify (quasi)-projective manifolds $X$ admitting infinite cyclic coverings
which have the homotopy type of a (quasi)-projective manifold. 
Alternatively, identify quasi-projective groups $G$ having normal subgroups $N$, such that 
$G/N\cong\Z$, which are again quasi-projective.
\end{problem}

A very natural idea is to start with a polynomial mapping $f:\C^n\to\C$,
and analyze the restriction $f: X\to\C^*$, where $X=\C^n\setminus V(f)$. 
The bifurcation set $\L=\L_f$ of $f$ is the minimal set to be removed so that 
$f: X\setminus f^{-1}(\L)\to \C^*\setminus \L$ is a locally trivial fibration. 
Then $\L$ contains $C=C_f$ the set of critical values of $f$ (in general
only as a proper subset). 

Now suppose that $f$ has finitely many singular points, and it behaves like 
a proper map, in particular the critical set is the bifurcation set of $f$. 
If $n\ge 3$, then one can obtain, by standard arguments, an exact sequence of groups
\[
1\to\pi_1(F_{a})\to\pi_1(X)\to\Z\to 1,
\]
for any fiber $F_{a}=f^{-1}(a), a\in\C^*$ of $f$, whether smooth or singular,
see~\cite{dps:08a,shi:03}.

In the present paper a different approach will be presented. 
Suppose $f$ and $g$ are two polynomial mappings $\C^n\to\C$, 
and consider the restrictions $f_{|T}: T\to\C^*$, where $T=\C^n\setminus V(f)$,
and $f_{|X}: X\to\C^*$, where $X=T\setminus V(g)$. 

If $f_{|T}: T\to\C^*$ is a trivial fiber bundle, and $V(g)$ is transverse (in the stratified sense) 
to the fibers $f_{|T}^{-1}(a)$ with $a\in\C^*\setminus \L$, then $f_{|X_\L}: X_\L\to\C^*\setminus \L$,
$X_\L:=X\setminus f^{-1}(\L)$, is again a fiber bundle ($\L$ is the bifurcation set of $f: X\to\C^*$). 

Let $\S$ be the union of the strata in $V(g)$ which are not transverse to $f^{-1}(\L)$, 
and denote its dimension by $s:=\dim\S$. Then, using Dold-Thom fibration theory one can 
see that the map $f: X\to\C^*$ behaves like a fibration up to dimension~$n-s-1$ (details will be
given in Section~\ref{sec:fht}). 

Now suppose $X$ is aspherical and $s=0$. It follows that all the fibers 
$F_{a}=f^{-1}(a)$ share the same homotopy type up to dimension $(n-2)$, 
in particular, if $n\ge 4$, the same fundamental group. Furthermore, if 
the monodromy of $f: X\to\C^*$ about the points in $\L$ is trivial, 
we then have an exact sequence $1\to\pi_1(F)\to\pi_1(X)\to\Z\to 1$.
We will show that indeed that is the case for the constructions of interest to us, 
by analyzing the homotopy short exact sequence 
$1\to\pi_1(F)\to\pi_1(X_\L)\to\pi_1(\C^*\setminus \L)\to 1$.

The paper is organized as follows: in the first two sections the main objects are introduced.
More specifically the group theoretical objects, such as 
Artin groups, Artin kernels, and their presentations, are described in Section~\ref{sec:ak}. 
The geometrical objects such as hyperplane and toric arrangements are described in Section~\ref{sec:hac}
together with their homological properties. The main results extending Theorems~\ref{thm1-intro}
and~\ref{thmht-intro} are stated in Section~\ref{sec-mainthm} and their proofs will be shown in
the following two sections. Section~\ref{sec:gc} describes the arrangements realizing the quasi-projective
groups of Section~\ref{sec:ak}. Finally, the homotopical and homological properties of these groups are 
discussed in Section~\ref{sec:fht} in terms of arrangements using quasifibration theory.

We would like to thank to the referee for 
his/her
adequate comments and remarks.

\section{Artin kernels}\label{sec:ak}

Let $\G$ be a finite simplicial graph, i.e., $\G:=(V_\G,E_\G)$, where $V_\G$ is 
a finite set, say $\{0,1,\dots,s\}$, and $E_\G\subset\{A\subset V_\G\mid\# A=2\}$; 
let $L_\G$ be the \emph{flag simplicial complex} generated by $\G$, where $J:=\{j_1,\dots,j_k\}$ 
is a simplex of $L_\G$ if and only if the complete graph on the vertices $J$
is a subgraph of~$\G$. The {\em right-angled Artin group} $A_{\G}$ associated with $\G$ is 
the group with generators $\s_1,\dots, \s_s$ and relations $\s_i\s_j =\s_j\s_i$, one for 
each edge $\{i,j\}\in E_\G$ of $\G$.

The Eilenberg-MacLane space $K(A_{\G},1)$ has two remarkable incarnations as toric
spaces. First, consider the real $(s+1)$-torus $T={(\SS^1)}^{s+1}$ with its standard CW-complex
structure, where $k$-cells are in bijection with $k$-subsets of $\{0,1,\dots,s\}$. 
From a simplicial complex $L=L_\G$ one constructs the subcomplex $T_L$
of $T$, by retaining only the cells corresponding to the simplices of $L$.  
It is well known that $T_L$ is a $K(A_{\G},1)$-space~\cite{chd:95}.
In the language of~\cite{ds:07}, $T_{L}$ is the \emph{moment-angle complex} $Z_L({\SS^1})$,
obtained by gluing subtori of $T$ according to the simplicial structure of $L$.
A similar gluing process, but using the complex $(s+1)$-torus $\T^{s+1}={(\C^*)}^{s+1}$
and complex subtori accordingly, produces $Z_L({\C^*})$, which we denote
by $\T_{L}$. 

The properties of the construction $Z_L$, see~\cite{ds:07}, ensure
that $T_{L}$ and $\T_{L}$ have the same homotopy type. It also follows
that $T_{L*M}=T_L\times T_M$, and respectively $\T_{L*M}=\T_L\times \T_M$,
where $L*M$ is the join of $L$ and $M$. 

\begin{example}
Let us consider two standard examples of graphs on $s+1$ vertices, namely,
$K_{s+1}$ the complete graph, and $\bar{K}_{s+1}$ the graph with no edges. 
The flag complexes are respectively an $s$-dimensional simplex 
and $\bar{K}_{s+1}$ itself. As for the associated right-angled Artin groups, 
$A_{K_{s+1}}\cong \Z^{s+1}$ and $A_{\bar{K}_{s+1}}\cong \F_{s+1}$, the free group of rank $s+1$. 
The toric complexes are $T_{K_{s+1}}=T\simeq \T=\T_{K_{s+1}}$ and 
$T_{\bar{K}_{s+1}}=\bigvee^{s+1} \SS^1\simeq\bigvee^{s+1}\C^*=\T_{\bar{K}_{s+1}}$ 
respectively. Note that $\bigvee^{s+1}\C^*$ has the 
homotopy type of $\C$ punctured $s+1$ times. We denote by $\C^*_{s+1}$ an 
$(s+1)$-punctured $\C$. Thus we have $\T_{\bar{K}_{s+1}}\simeq\C^*_{s+1}$.
\end{example}

\begin{example}
Let $\bar{\mathbf{n}}:=(n_0,\dots,n_r)$ be an $(r+1)$-tuple of positive integers, $n_i\ge 1$. 
Denote by $\bar{K}_{\bar{\mathbf{n}}}$ the 
multipartite graph defined by the join $\bar{K}_{n_0}*\dots *\bar{K}_{n_r}$.
The associated flag complex $L=L_{\bar{K}_{\bar{\mathbf{n}}}}$ has the homotopy type
of a wedge of $m:=\prod_{i=0}^r(n_i-1)$ spheres $\bigvee^m  \SS^{r}$. 
The Artin group $A_{\bar{K}_{\bar{\mathbf{n}}}}$ is the product 
$\F_{n_0}\times\dots\times\F_{n_r}$ of free groups, and 
$\T_L\simeq \C^*_{n_0}\times\dots\times\C^*_{n_r}$. 
\end{example}

We will consider here a class of subgroups of the Artin group $A_{\G}$,
generalizing the Bestvina-Brady groups of~\cite{BB:97}. An \emph{Artin kernel} is 
nothing but the kernel of an epimorphism $\chi: A_{\G}\to \Z$, see~\cite{ps:09}.
More precisely, given an epimorphism $\chi: A_{\G}\to \Z$, we denote its kernel by $N_{\G}^{\chi}$. 
The \emph{living} flag complex $L_{\G}^{\chi}$ associated to~$\chi$ is the flag complex of
the full subgraph of $\G$ whose vertices~$v$ are those such that $\chi(v)\neq 0$.

\begin{thm}[\cite{BB:97,MR1610579,bxgz:99,ps:09}]\label{thm:bxgz}
The Artin kernel $N_{\G}^{\chi}$ is finitely presented if and only if $L_{\G}^{\chi}$ is simply-connected. 
Moreover, $N_{\G}^{\chi}$ is of type $\text{\rm F}_r$ (respectively $\text{\rm FP}_r$) 
if and only if the flag complex $L_{\G}^{\chi}$ is $(r-1)$-connected (respectively $(r-1)$-acyclic).
In addition, $N_{\G}^{\chi}$ is of type $\text{\rm FP}_r$ if and only if 
$\dim_{\K}H_{\le r}(N_{\G}^{\chi},\K)<\infty$, for any field $\K$. 
\end{thm}

In this paper we will always assume that all integers $\chi(\s_i)$ are non-zero. This implies that 
$L_{\G}^{\chi}$ in Theorem~\ref{thm:bxgz} is the usual flag complex $L_{\G}$ on $\G$.
Let $\mathbb{T}_\Gamma^\chi\to\mathbb{T}_\Gamma$ be the infinite cyclic cover determined by~$\chi$.
Note that $\pi_1(\mathbb{T}_\Gamma^\chi)\cong N_{\Gamma}^{\chi}$.

The particular case where $\chi$ is the epimorphism $\chi:A_{\G}\to \Z$ defined by
$\chi(\s_i):=1$, the kernel $N_\G:=\ker\chi$ is known as the {\em Bestvina-Brady group} 
associated with $\G$. As noted above, this group is finitely presented if and only if $L$ is simply-connected. 
In that case, $N_{\G}$ admits a commutator-relators presentation, as it is shown in~\cite[Corollary~2.3]{ps:07}
(which follows essentially from~\cite[Corollary~3]{DicksLeary-Artin}).
As a consequence, one immediately obtains that $b_1(N_{\G})=|V_\G|-1$ 
and $b_2(N_{\G})=|E_\G|-|V_\G|+1$.

The Artin kernels that interest us most here are those associated with the multipartite
graph $\bar{K}_{\bar{\mathbf{n}}}$. For this special case, we introduce now some notation.
After a suitable labeling, we denote the set of vertices of $\bar{K}_{\bar{\mathbf{n}}}$ by
$\bigcup_{k=0}^r\{(k,i)\mid 1\leq i\leq n_k \}$ where $(k,i)$ is a vertex of $\bar{K}_{n_k}$. 
The corresponding generators of $A_{\bar{K}_{\bar{\mathbf{n}}}}$ are denoted by $\sigma_{k,i}$.

Let $\mathbf{d}:=(d_0,\dots,d_r)\in\Z^r$ and consider the epimorphism
$\chi_{\mathbf{d}}:A_{\bar{K}_{\bar{\mathbf{n}}}}\to\Z$ defined by~$\chi_{\mathbf{d}}(\sigma_{k,i}):=d_k$. 

\begin{cor}
The Artin kernel $N_{\bar{\mathbf{n}}}^{\mathbf{d}}:=N_{\bar{K}_{\bar{\mathbf{n}}}}^{\chi_{\mathbf{d}}}$
is of type $\text{\rm F}_{r}$, but not $\text{\rm FP}_{r+1}$, if all $n_i>1$. 
\end{cor}

\begin{proof}
Since $L_{\bar{K}_{\bar{\mathbf{n}}}}\simeq\bigvee^m \SS^{r}$ is $(r-1)$-connected, 
but $H_{r}(L_{\bar{K}_{\bar{\mathbf{n}}}})\neq 0$, it follows from Theorem~\ref{thm:bxgz} that 
$N_{\bar{\mathbf{n}}}^{\mathbf{d}}$
is of type $\text{F}_{r}$, but not $\text{FP}_{r+1}$, if all $n_i>1$. 
\end{proof}

Note that, as soon as $r\ge 2$, the Artin kernel $N_{\bar{\mathbf{n}}}^{\mathbf{d}}$ is finitely 
presented and thus the $\text{F}_m$ and $\text{FP}_m$ properties coincide in this case.
For $N=N_{\bar{K}_{\bar{\mathbf{n}}}}$ one obtains:
\[
b_1(N)=\sum_{i=0}^r n_i -1,\quad \quad b_2(N)= \sum_{0\le i<j\le r} n_i n_j 
-\sum_{i=0}^r n_i +1.
\]
Given a group $G$ of type $\text{FP}_m$, for some integer $m$, we recall the definition of the 
Poincar\'e series of $G$, as the generating function of its Betti numbers, that is,
$
P_G(t):=\sum_{k=0}^{\infty} b_k(G) t^k.
$
From the results of~\cite{ps:09} it follows that the truncated Poincar\'e 
polynomials of $A_{\G}$ and $N_{\G}$, where~$\Gamma$ is a multipartite graph, are related as 
\[
P_{A_{\G}}(t)\equiv (1+t)P_{N_{\G}}(t)\mod t^{r+1}.
\]
It is easily seen that the $r$-th truncation of the Poincar\'e polynomial of 
$A=A_{\bar{K}_{\bar{\mathbf{n}}}}$ is given by 
\[
P_{A}(t)\equiv \prod_{i=0}^r (1+n_i t) \mod t^{r+1}.
\]
This gives the Betti numbers $b_k$ of $N=N_{\bar{K}_{\bar{\mathbf{n}}}}$, 
in the range $0\le k<r$:
\[
b_k(N)=\sum_{p=0}^k (-1)^{k-p} 
\sum_{0\le i_1<\dots<i_p\le r}n_{i_1}\dots n_{i_p}.
\]

\begin{nothing}\textbf{Finite presentations of $N^\chi$.}\label{fp-chi}
A finite presentation for $N_\Gamma$ can be obtained using the Reidemeister-Schreier method;
a presentation of  $N_\Gamma$ for arbitrary $\Gamma$ can be found in~\cite{DicksLeary-Artin}.
Let us compute presentations for $N^{\chi}=N_{\bar{\mathbf{n}}}^{{\mathbf{d}}}$,
where $r\ge 2$,  $n_0,n_1,\dots,n_r\ge 2$ and $\gcd(d_0,d_1,\dots,d_r)=1$.
Fix $a_k\in\Z$ so that $\sum_{k=0}^r a_k d_k=1$ and define
$\tau:=\prod_{k=0}^r \s_{k,1}^{a_k}$; note that $\chi(\tau)=1$. Since $\s_{k,1},\s_{\ell,1}$ 
commute whenever $k\neq\ell$,
note that $\tau^m=\prod_{k=0}^r \s_{k,1}^{m a_k}$, $m\in\Z$. We can add $\tau$ to the generators
of $A_{\bar{K}_{\bar{\mathbf{n}}}}$, adding the following relations (some of them are redundant but useful):
\begin{equation}\label{eq-rels-tau}
\tau=\prod_{k=0}^r \s_{k,1}^{a_k},\quad [\tau,\sigma_{k,1}]=1,\quad k=0,\dots,r. 
\end{equation}

Applying the Reidemeister-Schreier method, a set of generators of $N^{\chi}$ is given by
\[
\a_{k,i;m}:=\tau^m \s_{k,i} \tau^{-m-d_{k}}, 0\le k\le r,\, 1\leq i\leq n_k,\, m\in\Z.
\]
The commutativity relations in \eqref{eq-rels-tau} imply that
$\a_{k,1;m}$ does not depend on $m$ and hence we denote them by $\b_{k}$.
The remaining relation in \eqref{eq-rels-tau} gives the following
one for $N^\chi$:
\begin{equation}\label{eq-rels-0}
\prod_{k=0}^r \b_k^{a_k}=1.
\end{equation}
The relations $[\s_{k,i},\s_{\ell,j}]=1$, $k\neq\ell$ give 
the following ones in the kernel $N^{\chi}$
\begin{equation}
\label{eq:rels-an}
\a_{k,i;m}\a_{\ell,j;m+d_{k}}=\a_{\ell,j;m}\a_{k,i;m+d_{\ell}},\text{ if } k\neq\ell. 
\end{equation}
If $i=j=1$, \eqref{eq:rels-an} simplifies as:
\begin{equation}\label{eq-rels-comm}
[\b_k,\b_\ell]=1, k\neq\ell.
\end{equation}
If $i=1<j$, \eqref{eq:rels-an} simplifies as:
\begin{equation}
\label{eq:rels-anb}
\a_{\ell,j;m+d_k}=\b_k^{-1}\a_{\ell,j;m}\b_k,\quad 
\forall m\in\Z,\ k\in\{0,\dots,r\},\ \ell\neq k,\ j\in\{2,\dots,n_\ell\}. 
\end{equation}
Let us denote $e_\ell:=\gcd\{d_k\mid k\neq\ell\}$ and choose integers $a_{k,\ell}$, $a_{k,k}:=0$,
such that 
$$
e_\ell=\sum_{k=0}^r a_{k,\ell} d_k.
$$
Clearly, using relations~\eqref{eq:rels-anb}, we can reduce the infinite
set of generators $\{\a_{k,i;,m}\mid m\in\Z,\,k\in\{0,\dots,r\},1\leq i\leq n_k\}$ to the finite subset 
\begin{equation}\label{eq-gens}
\{\b_k\mid 1\leq k\leq r\}\cup\{\a_{k,j;p}\mid 0\le p<e_k,\,0\leq k\leq r,\,1<j\leq n_k\}. 
\end{equation}
Let us denote $\g_\ell:=\prod_{k=0}^r \beta_k^{a_{k,\ell}}$. Then, we obtain:
\begin{equation}
\label{eq:rels-anc}
\a_{\ell,j;m+e_\ell}=\g_\ell^{-1}\a_{\ell,j;m}\g_\ell,\quad 
\forall m\in\Z,\ \ell\in\{1,\dots,r\},
\ j\in\{2,\dots,n_\ell\}. 
\end{equation}
Denote $f_{k,\ell}:=\frac{d_k}{e_\ell}$, $k\neq\ell$.
Combining~\eqref{eq:rels-anb}-\eqref{eq:rels-anc}, we obtain:
\begin{equation}\label{eq-rels-kl}
[\a_{\ell,j;m},\g_\ell^{f_{k,\ell}}\b_k^{-1}]=1\quad
k\in\{0,\dots,r\},\ \ell\neq k,\ j\in\{2,\dots,n_\ell\},\ 
m\in\{0,\dots,e_k-1\}.
\end{equation}
Consider now the relations \eqref{eq:rels-an} for $i,j>1$:
\begin{equation}
\label{eq:rels-an-i}
\a_{k,i;m}\b_k^{-1}\a_{\ell,j;m}\b_k=\a_{\ell,j;m}\b_\ell^{-1}\a_{k,i;m}\b_\ell,\text{ if } k\neq\ell. 
\end{equation}
Given $k\neq\ell$, fix $m\neq k,\ell$ such that $d_m$ is minimal; this is possible since $r\geq 2$.
Then it is enough to consider relations \eqref{eq:rels-an-i} for $m\in\{0,\dots,d_m-1\}$. 

The presentation of $N^\chi$ admits~\eqref{eq-gens} as system of generators and~\eqref{eq-rels-0}, 
\eqref{eq-rels-comm}, \eqref{eq-rels-kl}, and~\eqref{eq:rels-an-i} as a system of relators.
\end{nothing}

It seems plausible that the approach of Dicks and Leary from~\cite{DicksLeary-Artin},
further simplified in~\cite{ps:07}, can be used to produce finite presentations
for the Artin kernels $N^{\chi}_{\Gamma}$ that would involve the combinatorics
of the flag complex of $L_{\Gamma}$ in a precise geometric way, but we have not 
been able to determine such presentations.

\begin{example}\label{ex-d1}
If $d_0=1$, we can choose $\tau=\s_{0,1}$, i.e. $a_0=1$ and $a_k=0$
if $k>0$. We get $\beta_0=1$, $e_k=1$ if $k>0$, and hence $\g_k=1$ if $k>0$.
Then $N^{\chi}$ is generated by 
$\a_{k,i;m}=\s_{0,1}^m \s_{k,i} \s_{0,1}^{-m-d_{k}}$, 
for 
$$
\{(0,i;m)\mid 1<i\leq n_0,\ 0\leq m<e_0\}\cup\{(k,i;0)\mid 0<k\leq r,\ 1\leq i\leq n_k\}.
$$
Let us denote $\mu_{i,m}:=\a_{0,i;m}$ and $\nu_{k,i}:=\a_{k,i;0}$ ($k>0$).
Note also that if $k\neq\ell$ and both are different from~$0$ then, we can choose $m=1$.
The relations are:
\begin{align*}
[\nu_{k,1},\nu_{\ell,1}]&=1,& 0<k<\ell,\\
[\mu_{i,m},\gamma_0^{\frac{d_k}{e_0}}\nu_{k,1}^{-1}]&=1,\quad& 1<i\leq n_0,\ 0\leq m<e_0,\ 0<k\leq r,\\ 
[\nu_{k,i},\nu_{\ell,j}]&=1,\quad& 0<k,\ell\leq r,\ k\neq\ell,\ 1\leq i\leq n_k,\  1\leq j\leq n_\ell,\\
[\mu_{j,m},\nu_{k,i}\nu_{k,1}^{-1}]&=1&0<k\leq r,\ 1\leq i\leq n_k,\ 1<j\leq n_0,\ 0\leq m<d_0,
\end{align*} 
where $\mu_{j,m+e_1}:=\gamma_0^{-1}\mu_{j,m}\gamma_0$.
\end{example}

\section{Hypersurface arrangements and their complements}\label{sec:hac}

In this section we introduce the arrangements of hypersurfaces that will be needed 
in the sequel. We will consider two classes: hyperplane arrangements and toric
arrangements. In each case, we will explain how the homology of complements
to such arrangements may be computed in terms of their combinatorics. We will end
with computations of the Poincar{\'e} polynomials for the arrangements that we will
later relate with Artin kernels. 

\subsection{Hyperplane arrangements}
\mbox{}

An arrangement of hyperplanes is a collection $\A$ of hyperplanes in a projective,
or affine space. An arrangement $\A$ in $\P^r$ may be viewed as a central arrangement
in $\C^{r+1}$. If $M$ is the complement of $\A$ in  $\P^r$, and $M'$ the complement  
in $\C^{r+1}$, we have that $M'$ is homeomorphic to $M\times\C^*$. It follows that
$P(M',t)=(1+t)P(M,t)$, and $\pi_1(M')=\pi_1(M)\times\Z$.

The Poincar{\'e} polynomial of a complement to a hyperplane arrangement $\A$ is completely 
determined by its intersection poset $L(\A)$, which consists of all non-empty intersections 
among hyperplanes in $\A$ ordered by reversed inclusion and ranked by codimension. 

More precisely, if $\A$ is an affine hyperplane arrangement, the characteristic
polynomial of the intersection poset $L(\A)$ is defined by
$
\chi(\A,q):=\sum_{S\in L(\A)}\mu(S)q^{\dim S},
$
where $\mu:L(\A)\to\Z$ is the M\"obius function of the ranked poset $L(\A)$, see~\cite{orte:92},
which satisfies
\[
P(M,t)=t^{r+1}\chi(\A,-t^{-1}).
\]

The best understood arrangements are those of supersolvable type, see~\cite{orte:92}. 
The complement $M$ of such an arrangement seats at the top of a tower of linear fiber 
bundles with fiber a punctured complex plane. In particular $M$ is a $K(G,1)$ space 
for $G:=\pi_1(M)$. A more general class, comprising both supersolvable and generic 
arrangements, is that of hypersolvable arrangements, see~\cite{jp:02}. The group $G$ of a 
hypersolvable arrangement $\A$ is still of type $\text{F}$, as it is also the group of a 
supersolvable arrangement $\hat{\A}$ to which $\A$ deforms. 

A class of arrangements that contains both hypersolvable and non-hypersolvable examples 
is that of graphic arrangements. We refer to~\cite{ps:02, ps:06a} for details and further discussion.
One associates to a graph $\GG=(\VV,\EE)$ on a set 
of vertices $\VV:=\{0,\dots,r\}$, the arrangement $\A_{\GG}$ in $\C^{r+1}$ 
of hyperplanes $H_{i,j}=\{w_i=w_j\}$ indexed by all edges $\{i,j\}$ of $\GG$.
Any $\A_{\GG}$ is a sub-arrangement of the braid arrangement, 
which corresponds to the complete graph $K_{r+1}$. Denote by $c_p$ the number of complete 
subgraphs $K_{p+1}$ of $\GG$. 

A graphic arrangement $\A_{\GG}$ is supersolvable if and only if $\GG$ is a chordal
graph (i.e. every circuit of length $\ge 4$ in $\GG$ has a chord).
For example, an arrangement $\A_{\GG}$ with $c_2=0$ is hypersolvable. 
Suppose $\GG$ contains no complete graph $K_4$ as a subgraph, that is $c_3=0$. 
A sufficient condition for such an arrangement $\A_{\GG}$ to be 
non-hypersolvable is that $0<c_1\le 2c_2$.

\begin{example}
Let $\GG=W_{r}$ be the (wheel) graph on $r+1$ vertices obtained by coning an $r$-cycle. 
In this case, $\A_{W_{r}}$ is defined, after the linear change of coordinates 
$z_0=w_0,z_i=w_i-w_0,1\le i\le r$ by 
$
z_1 \dots z_r (z_1-z_2)\dots (z_{r-1}-z_r)(z_r-z_1).
$
Then $c_1=2r$, $c_2=r$, and $c_3=0$, if $r>3$.
Thus $c_1=2c_2$, and so $\A_{W_{r}}$ is non-hypersolvable and decomposable,
see~\cite{ps:06a} for a definition. 
We shall see that the group of the arrangement $\A_{W_{r}}$ is not of type $\text{F}_r$. 
\end{example}

\subsection{Toric arrangements}
\mbox{}

Let $\nu$ be a character of the complex algebraic torus $\T^{r+1}=(\C^*)^{r+1}$ and $a\in\C^*$. 
The pair $(\nu, a)$ defines a hypersurface
$H_{\nu,a}=\{\mathbf{x}\in\T^{r+1}\mid \nu(\mathbf{x})=a\}$. 
Note that if $\nu(\mathbf{x})=x_0^{d_0}\cdot\ldots\cdot x_r^{d_r}$,
then $H_{\nu,a}$ has $d:=\gcd(d_0,\dots,d_r)$ connected components, which are hypertori.

A finite set $\A$ of such hypertori in $\T^{r+1}$ is called a toric arrangement.
Much of the combinatorial and topological theory of such arrangements 
parallels that of hyperplane arrangements, see~\cite{dcp:11}. 
The de Rham cohomology ring of the complement was determined in~\cite{dcp:05}.

Furthermore, the Poincar{\'e} polynomial of the complement 
$M=\T^{r+1}\setminus V(\A)$ and the characteristic polynomial 
of the intersection poset $L(\A)$ of $\A$, 
are related by 
$$P(M,t)=(-t)^{r+1}\chi(\A,-t^{-1}-1),$$
as shown in~\cite[Corollary~5.12]{moc}.

\begin{rem}
Let $\mathcal{H}$ be a hyperplane arrangement  in $\C^{r+1}$
defined by equations of the form $z_i-az_j=0$, with $a\in\C^*$.
Consider the hyperplane arrangement $\A$ obtained 
by adding the coordinate hyperplanes to $\mathcal{H}$.
If $\mathcal{T}$ is the toric arrangement in $\T^{r+1}$ 
obtained by the restriction of $\mathcal{H}$, then the complement of $\mathcal{T}$ in $\T^{r+1}$
may be seen as the complement of $\A$ in $\C^{r+1}$.
The equality of Poincar{\'e} polynomials gives the following relation between 
characteristic polynomials
\[
\chi(\A,q)=(-1)^{r+1}\chi(\mathcal{T},q-1). 
\]
\end{rem}

\subsection{General position toric arrangements}
\mbox{}

We focus now on a particular class of toric arrangements, which consists of 
general position toric arrangements and some mild deformations of them.
Our goal is to compute the Poincar\'e polynomials of such arrangements.

We start by analyzing an example illustrating all the features of the general
case. 

\begin{example}
Let $\A_a=\{H_{0,a}, H_1,\dots,H_{r}\}$ be the toric arrangement in $\T^r$  
defined by the hypertori $H_0=H_{0,a}=\{x_1\cdot\ldots\cdot x_r=a\}$ and 
$H_i=\{x_i=1\}$, $1\le i\le r$, for some $a\in\C^*$. 

Note that the hypertori in $\A_a$ intersect in general position, except maybe if 
we consider the intersection of all $r+1$ of them. Indeed, if $H_I$ denotes the 
intersection $\bigcap_{i\in I} H_i$, and $I$ is a proper subset of $\{0,\dots,r\}$, then 
$H_I$ is a connected torus of codimension $|I|$. Now $\bigcap_{i=0}^r H_i$ is
either empty if $a\neq 1$, or the point $(1,\dots,1)\in\T$. It follows that 
the intersection poset $L(\A_a)$ is either the truncation of the lattice of subsets
of $\{0,\dots,r+1\}$ (maximal element removed), if $a\neq 1$, or the obvious collapsing
of this truncation for $a=1$. 
The calculation of the characteristic polynomial of $\A_a$ is then immediate:
\[
\chi(\A_a,q)=
\begin{cases}
\sum_{k=0}^r (-1)^k \binom{r+1}{k} q^{r-k}& \text{ if } a\neq 1, \\
\sum_{k=0}^r (-1)^k \binom{r+1}{k} q^{r-k}-(-1)^{r+1}& \text{ if } a=1.
\end{cases}
\]
After evaluating at $q=-\frac{1+t}{t}$, we obtain the Poincar{\'e} polynomials
\[
P(M(\A_a),t)=
\begin{cases}
\dfrac{(2t+1)^{r+1}-t^{r+1}}{t+1}& \text{ if } a\neq 1,\\
&\\
(2t+1)\dfrac{(2t+1)^r-t^r}{t+1}& \text{ if } a=1.
\end{cases}
\]
Note that the complements $M(\A_{a\neq 1})$ differ from $M(\A_{1})$
only in their top Betti number $b_r$, as $P(M(\A_a),t)-P(M(\A_1),t)=t^r$. 
\end{example}

We now treat the general case. We work in $\T^{r+1}$ and fix
a character $\nu$ such that $\nu(\mathbf{x})=x_0^{d_0}\cdot\ldots\cdot x_r^{d_r}$ 
with $\gcd(d_0,\dots,d_r)=1$.
Consider the connected $r$-torus $\T_{\nu,a}=\{\nu(x)=a\}$ 
in $\T^{r+1}$ for $a\in\C^*$. 
Let $\mathbf{n}:=(n_0,n_1,\dots,n_r)$ and $m_i:=n_i-1$.
We want to compute the Poincar{\'e} polynomial of the complement 
in $\T_{\nu,a}$ of the arrangement $\A_{\nu,a}^{\mathbf{n}}(\boldsymbol{\alpha})$ 
consisting of the hypertori
$H_i^{j}=\{x_i=\a_{i,j}\}$, $0\le i\le r$, $1\le j\le m_i$.
As we shall see, $\A_{\nu,a}^{\mathbf{n}}(\boldsymbol{\alpha})$ depends on 
the data $\boldsymbol{\alpha}:=(\a_{i,j})$.
Denote by $\Lambda:=\{\nu(\a_{0,j_0},\dots,\a_{r,j_r})\mid 1\leq j_i\le m_i,\ 0\leq i\leq r\}$.
We will also use the following notation: $[r]=\{0,\dots,r\}$, and for any $I\subseteq [r]$, 
$d_I:=\gcd\{d_i, i\in I\}$, $m_I:=\prod_{i\in I} m_{i}$ and $\bar{I}:=[r]\setminus I$.

\begin{prop}\label{prop:poinc}
Let $M_a:=M(\A_{\nu,a}^{\mathbf{n}}(\boldsymbol{\alpha}))$ be
the complement of $\A_{\nu,a}^{\mathbf{n}}(\boldsymbol{\alpha})$ in $\T_{\nu,a}$.
Suppose $a\not\in\Lambda$. Then the Poincar{\'e} polynomial
of the complement  is equal to
\[
P(M_a,t)=\sum_{k=0}^r t^k(1+t)^{r-k}\sum_{I\subseteq[r], |I|=k}d_{\bar{I}} m_I.
\]
In particular, the $i$-th Betti number of $M_a$ is 
\[
b_i(M_a)=\sum_{l=r-i}^r\binom{l}{r-i}\sum_{|I|=r-l}d_{\bar{I}} m_I.
\] 
Furthermore, if $\l\in\Lambda$ then 
$P(M_{\l},t)\neq P(M_a,t)$. More precisely, we have
\begin{equation}
\label{eq-betti}
b_i(M_a)=b_i(M_{\l})\ \forall i<r, \text{ and } b_r(M_a)>b_r(M_{\l}).
\end{equation}
\end{prop}

\begin{proof}
We compute the characteristic polynomial $\chi(\A,q)$ of the toric arrangement 
$\A=\A_{\nu,a}^{\mathbf{n}}(\a)$.
We claim that the hypertori $H_i^{j}$ intersect in general position, except maybe if 
we consider intersections of $r+1$ of them. 

First note that $H_i^{j'}$ and $H_i^{j''}$ never intersect if $j'\neq j''$.
Then denote by $H_I^{\mathbf{j}}$ the intersection 
\[
H_I^{\mathbf{j}}=\T_{\nu,a}\cap\bigcap_{i\in I}  H_i^{j_i}, 
\text{ for } I\subseteq[r],\ \mathbf{j}=(j_i)_{i\in I}\in\prod_{i\in I}\{1,\dots,m_i\}.
\]
If $|I|\le r$, then it can be easily seen that $H_I^{\mathbf{j}}$ is non-empty and 
of codimension $r-|I|$, and it consists of $d_{\bar{I}}$ connected components. 
If $a\not\in\Lambda$, then $H_{[r]}^{\mathbf{j}}$ is empty for all $\mathbf{j}$, 
and so it does not appear in $L(\A)$. 
Moreover, the M\"obius function is readily computed as $\mu(H_I^{\mathbf{j}})=(-1)^{|I|}$. 
 
Summing up the M\"obius function contributions for all 
the connected components of $H_I^{\mathbf{j}}$ one obtains
\[
\chi(\A,q)=\sum_{k=0}^r (-1)^k q^{r-k}\sum_{I\subseteq[r], |I|=k}d_{\bar{I}} m_I,
\]
which, after the substitution $q=-t^{-1}-1$, gives the Poincar{\'e} polynomial of 
the $M_a$. 
The formula for the Betti numbers follows after a routine rearrangement of the sum. 

Consider now the case $\l\in\Lambda$. It is readily seen that the positive dimensional intersections
$H_I^{\mathbf{j}}$ have the same description as in the generic case.  
Thus the coefficient of $q^{r-k}$ in $\chi(\A,q)$ will stay the same if $k<r$, hence the equality 
in~\eqref{eq-betti} follows.

The only difference may appear in the independent term, which comes from codimension $r$, i.e.
points. To that coefficient only the cardinality of $H_I^{\mathbf{j}}$ with $|I|$ equals $r$ or $r+1$ 
may contribute. By hypothesis $H=H_{[r]}^{\mathbf{j}}\neq\emptyset$ for some~$\mathbf{j}$, whereas
for $a\notin \Lambda$, $H_{[r]}^{\mathbf{j}}=\emptyset$, $\forall \mathbf{j}$. Such $H$ consists in 
one point and $\mu(H)=r(-1)^r$. Also note that $H$ results from the collapsing of $r+1$ points
(each one in $H_{\bar\iota}^{\mathbf{j}_i}$, where $\mathbf j_i$ is obtained by forgetting $j_i$ in $\mathbf j$).
If these $r+1$ points were different, they would contribute with $\mu=(-1)^r (r+1)$, which proves the
last assertion.
\end{proof}

\section{Main Theorems}
\label{sec-mainthm}
In this section we collect the main results of the paper. The proofs will be given in
the last two sections.
The quasi-projective Bestvina-Brady groups were determined in \cite{dps:08a}. 

\begin{thm}[\cite{dps:08a}]\label{thm:dps}
Let $\G$ be a finite graph. 
Suppose the Bestvina-Brady group $N_\G$ associated to $\G$ is quasi-projective. 
Then $\G$ is either a tree, or a  complete multipartite graph 
$K_{n_0,\dots, n_r}$, with either some $n_i=1$ or all $n_i\ge 2$ 
and $r\ge 2$. 
\end{thm}

We shall make this statement more precise, by showing that all quasi-projective groups 
$N_\G$ are in fact fundamental groups of hyperplane arrangement complements in $\P^r$.

\begin{thm}\label{thm-111}
Any quasi-projective Bestvina-Brady group $N_\G$ is a hyperplane arrangement group
for an arrangement in $\P^r$.
\end{thm}

In the first case, $N_\G$ is simply a free group of rank $v-1$ ($v$ is the number of vertices
of the tree) , which is known
to be the fundamental group of the complement in $\P^1$ of $v$ points.

The second case will interest us most. 
Let $\G=K_{\mathbf{n}}$, with $\mathbf{n}:=(n_0,\dots, n_r)$.
The first part is again easy. Assume  $n_i=1$ for $0\leq i\leq s$ 
and $n_i>1$ for $s< i\leq r$. Then it is readily seen that 
$N_{K_{\mathbf{n}}}\cong\Z^{s}\times\F_{n_{s+1}}\times\cdots\times\F_{n_r}$.
This group is clearly the fundamental group of the complement of
a hyperplane arrangement in $\C^r=\P^r\setminus\{z_0=0\}$ with equation
$$
z_i=\a_i,\quad 1\leq i\leq s,\qquad
z_i=\a_{i,j},\quad s+1\leq i\leq r,\quad 1\leq j\leq n_i,
$$
for generic choices of $\a_i,\a_{i,j}$. Taking a generic plane section
it can be seen as the fundamental group of the complement of
a line arrangement: Consider $r-s$ distinct directions in $\C^2$ and take $n_i$ lines
parallel to the ${s+i}^{\rm th}$ direction, and $s$ other lines in general 
position.

So only the case $N_{K_{\mathbf{n}}}$, with all $n_i\ge 2$ and $r\ge 2$, is left.
We will show in the next section that $N_{K_{\mathbf{n}}}$ is the fundamental
group of the complement in $\P^r$ to the hyperplane arrangement 
$\A_{a}^{\mathbf{n}}$ defined by the polynomial 
\[
z_0\cdot \ldots \cdot z_r \cdot \prod_{j=1}^{n_0-1}(a z_0-\a_{0,j} z_1) 
\cdot
\left( \prod_{i=1}^{r-1}\prod_{j=1}^{n_i-1}(z_{i}-\a_{i,j}z_{{i+1}}) \right)
\cdot \prod_{j=1}^{n_r-1}(z_r-\a_{r,j} z_0) ,
\]
A generic plane section provides a line arrangement with a precise combinatorics.

\begin{cor}\label{cor-111}
The group $N_{K_{\mathbf{n}}}$ can be realized as 
$\pi_1(\P^2\setminus \A_{\mathbf{n}})$, where $\A_{\mathbf{n}}$
is a line arrangement in $\P^2$ consisting of $(n_0+1)\cdot\ldots\cdot(n_r+1)$
lines, formed by $r$ lines in general position cutting out
a polygon with $r+1$ sides which have at vertex $i$ another $n_i-1$ 
lines that intersect transversally among themselves. 
\end{cor}

If $\G=K_{n_0,\dots, n_r}$ then the right-angled Artin group $A_\G$ 
is the product of free groups $\F_{n_0}\times\cdots\times\F_{n_r}$. 
Using the same arguments, one can extend this result to other
subgroups of $A_\G$, namely to the generalized Bestvina-Brady group
$N_{\mathbf{n}}^{\mathbf{d}}$, where $\mathbf{n}=(n_0,\dots,n_r)$,
and $\mathbf{d}=(d_0,\dots,d_r)$ with $\gcd(d_i)=1$.

\begin{thm}\label{thm-di}
Any generalized Bestvina-Brady group $N_{\mathbf{n}}^{\mathbf{d}}$ 
is the fundamental group of the complement to an algebraic hypersurface
in the weighted projective space $\P^{r}(\mathbf{d})$.
\end{thm}

In fact, it is more convenient to view this complement space as obtained
from an algebraic $r$-torus by removing $\sum_{i=0}^r (n_i-1)$ hypertori.
More precisely, we will consider the toric arrangement 
$\A_{\mathbf{n},a}^{\mathbf{d}}(\boldsymbol{\alpha})$ from Section~\ref{sec:hac}. 
In certain cases, as we shall see later, it is still possible to realize 
this toric arrangement complement as a hypersurface complement in $\P^r$, 
but in general this hypersurface is not the union of hyperplanes. 
Now recall that a $K(\pi,1)$ space for $N_{\mathbf{n}}^{\mathbf{d}}$ may be chosen
to be the infinite cyclic cover $\T_{K_{\mathbf{n}}}^{\chi^{\mathbf{d}}}$ of 
the product~$\T_{K_{\mathbf{n}}}$ of tori. Then, by putting together the 
geometric realizability of $N_{\mathbf{n}}^{\mathbf{d}}$ from Section~\ref{sec:gc} 
and its higher dimensional consequences from Section~\ref{sec:fht},
we obtain the following.

\begin{thm}\label{thm-ht}
Let $M=M(\A_{\mathbf{n},a}^{\mathbf{d}}(\a))$ be the complement of the toric arrangement defined above, 
where $r\ge 2$, $n_i>1$, and $a\not\in\Lambda$. Then we have the following
\begin{enumerate}
\item the $r$-skeleton of $\T_{K_{\mathbf{n}}}^{\chi^{\mathbf{d}}}$ has the homotopy
type of $M$.
\item the Artin kernel $N_{\mathbf{n}}^{\mathbf{d}}$ is isomorphic to $\pi_1(M)$.
\item the homotopy groups $\pi_i(M)$ vanish in the range $1<i<r$.
\item the homology groups $H_i(M)$ and $H_i(N_{\mathbf{n}}^{\mathbf{d}})$ are isomorphic
for $i\le r$.
\end{enumerate}
\end{thm}

\section{A geometric construction}\label{sec:gc}

Suppose $\T^{r+1}\setminus D$ is a hypersurface complement in a complex torus
of dimension~$r+1$. 
If $f\colon \T^{r+1}\setminus D\to\C^*$ is a polynomial function then let $\L_f\subset\C^*$ be
the bifurcation set of $f$, that is the smallest set $\L$ such that $f$ 
is a locally trivial fibration over the complement of $\L$.

In this section we will determine the bifurcation set $\L_f$ for certain monomial maps 
defined on $\T^{r+1}\setminus D$ for which $D$
is a union of hypertori of $\T^{r+1}$, that is a toric arrangement complement.

More precisely, given $(r+1)$-tuples $\mathbf{n}:=(n_i)_{i=0}^r$ and $\mathbf{d}:=(d_i)_{i=0}^r$ of 
positive integers,
with  $\gcd\mathbf{d}=1$, let $p:=p^{\mathbf{d}}$ 
denote the multiplication map: 
\begin{equation}
\label{eq-p}
p:\T^{r+1}\to\C^*,\qquad p(x_0,\dots,x_r):=x_0^{d_0}\cdot\ldots\cdot x_r^{d_r}.
\end{equation}
Now, let $f:=f_{\mathbf{n}}^{\mathbf{d}}:X\to\C^*$ denote the restriction of $p$ to the complement
$X:=\T^{r+1}\setminus D$ of the hypersurface 
\begin{multline}\label{eq-D}
D:=\bigcup_{(i,j)\in B} D_{i,j}\subset\T^{r+1}, \quad
D_{i,j}:=\{x_i=\a_{i,j}\}, \a_{i,j}\in \C^*,\\
B:=\{(i,j)\mid 1\leq j<n_i, 0\leq i\leq r\}. 
\end{multline}
Note that $X$ is homeomorphic to $\C^*_{n_0}\times\dots\times\C^*_{n_r}$, the toric 
complex $\T_{K_{\mathbf{n}}}$ 
associated to the complete multipartite graph $K_{\mathbf{n}}$; the analytic structure of
$X$ depend on the sets of $n_i-1$ points 
removed from each $\C^*$ factor.
In fact, $f=f_{\mathbf{n}}^{\mathbf{d}}$ also depends on them, but they are omitted for notational simplicity.

\begin{lem}
The map $f=f_{\mathbf{n}}^{\mathbf{d}}:X\to\C^*$ induces the group homomorphism
$\chi:=\chi_{\mathbf{n}}^{\mathbf{d}}:N^{\chi}\to\Z$.
\end{lem}

Let $a\in\C^*$. Then the fiber
$p^{-1}(a)=\{x_0^{d_0}\cdot\ldots\cdot x_r^{d_r}=a\}$ is a connected hypertorus
ambient isomorphic to the hypertorus $\{y_0=a\}$, via a monomial automorphism
$x_j\mapsto\prod_{j=0}^r y_i^{a_{ij}}$, determined by a
unimodular integer matrix $A=(a_{ij})_{0\le i,j\le r}$.   
Clearly the map $p$ is a trivial fiber bundle with fiber a connected torus.
Let $C:=\{\mathbf{j}=(j_0,\dots,j_r)\mid 1\leq j_i<n_i,0\leq i\leq r\}$.

\begin{prop}\label{prop-strat}
The bifurcation set $\L_f$ consists of the values 
$\{a_{\mathbf{j}}:=\a_{0,j_0}\cdot\ldots\cdot\a_{r,j_r}\}_{\mathbf{j}\in C}$.
For generic points $\{\a_{i,j}\}$ the map $f$ has exactly
$m:=(n_0-1)\cdot\ldots\cdot(n_r-1)$ distinct special fibers~$f^{-1}(a_{\mathbf{j}})$.
\end{prop}

\begin{proof}
The hypersurface $D$ determines a stratification of the ambient torus $\T^{r+1}$
as follows. For each $I\subset B$ consider the intersection 
$D_{I}=\bigcap_{(i,j)\in I} D_{i,j}$; the strata are defined as differences of these closed sets.
Then the top stratum is $X$ itself and the positive codimension strata are the intersections 
$D_{I}$ away from lower dimensional $D_{I'}$. 

Now $f:X\to\C^*$ is a restriction of $p$ which is a trivial fiber bundle. 
In order to apply the Thom Isotopy Lemma, it is enough to compactify $X$ as a subspace of 
(the normalization of)
$$
Z:=\{([x_0:\dots:x_r:y],t)\in \P^{r+1}\times \C^*\mid x_0^{d_0}\cdot\ldots \cdot x_r^{d_r}=ty^{d_{r+1}}\},
\quad d_{r+1}=\sum d_i,
$$
and extend $p$ to a proper map $Z\to\C^*$. Note that $p$ defines a trivial fibration on the strata at infinity.

In order to make $f$ into a fiber bundle we only have to remove those values $a\in\C^*$
(and their preimages from $X$) such that $p^{-1}(a)$ does not intersect transversally 
at least one stratum of the above stratification of $\T^{r+1}$. It is immediate that the fibers 
$p^{-1}(a)$ intersect $D_{I}$ transversally away from lower dimensional $D_{I'}$,
unless $D_{I}$ are already zero-dimensional. In that case, the stratum is 
simply a point which may lie or not on $p^{-1}(a)$. More precisely, $D_{I}$ 
is a point if and only if $I=I_{(j_0,\dots,j_r)}=\{(0,j_0),\dots,(r,j_r)\}$
for $(j_0,\dots,j_r)\in C$. Note that $I_{(j_0,\dots,j_r)}\subset f^{-1}(\a_{0,j_0}\cdot\ldots\cdot\a_{r,j_r})$
and hence, $\L_f$ is determined.

Finally, if $\{\a_{i,j}\}_{(i,j)\in B}$ are generic, then the values 
$a=\a_{0,j_0}\cdot\ldots\cdot\a_{r,j_r}$ are pairwise distinct.
\end{proof}

For the sake of simplicity, we will first treat the situation where $d_i=1$ for all $i$.
In that case we just write $f=f_{\mathbf{n}}$ and $p:\T^{r+1}\to\C^*$, $p(\mathbf{x})=x_0\cdot\ldots\cdot x_r$.
The fibers~$p^{-1}(a)=\{\mathbf{x}\in\T^{r+1}\mid x_0\cdot\ldots\cdot x_r=a\}$ of the product map 
are isomorphic to the $r$-torus $\T^r$
under the identifications 
\begin{equation}\label{eq-id1}
(x_0,\dots,x_r)\mapsto(x_1,\dots,x_r),\quad
(x_1,\dots,x_r)\mapsto\left(\frac{a}{x_1\cdot \ldots\cdot x_r},x_1,\ldots,x_r\right).
\end{equation}
Moreover one can identify $\T^r$ with $\P^r\setminus\{z_0\cdot \ldots \cdot z_r=0\}$,
via 
\begin{equation}\label{eq-id2}
(x_1,\dots,x_r)\!\mapsto\!\left[\!1:\prod_{j=1}^r x_j:\prod_{j=2}^r x_j:\dots:x_{r-1} x_r:x_r\!\right],\
[z_0:\dots:z_r]\!\mapsto\!\left(\frac{z_1}{z_2},\dots,\frac{z_r}{z_0}\right).
\end{equation}
Thus, composing the inverse of the identifications \eqref{eq-id1}-\eqref{eq-id2}, 
we obtain a diffeomorphism from $\P^r\setminus\{z_0\cdot \ldots \cdot z_r=0\}$ to $p^{-1}(a)$
provided by 
\begin{equation}\label{eq-id3}
[z_0:\dots:z_r]\mapsto\left(a\frac{z_0}{z_1},\frac{z_1}{z_2},\dots,\frac{z_r}{z_0}\right).
\end{equation}

\begin{lem}\label{lema-hyperplane}
Consider the map $f=f_{\mathbf{n}}:X\to\C^*$ as defined above. The fiber $F=f^{-1}(a)$ 
over $a\in\C^*$ is homeomorphic to the complement in $\P^r$ of the hyperplane arrangement $\A_{a}$ defined by the polynomial 
\begin{equation}\label{eq:arrangement}
z_0\cdot \ldots \cdot z_r \cdot \prod_{j=1}^{n_0-1}(a z_0-\a_{0,j} z_1) 
\cdot
\left( \prod_{i=1}^{r-1}\prod_{j=1}^{n_i-1}(z_{i}-\a_{i,j}z_{{i+1}}) \right)
\cdot \prod_{j=1}^{n_r-1}(z_r-\a_{r,j} z_0) ,
\end{equation}
\end{lem}

\begin{proof}
Under the identification 
$\{\mathbf{x}\in\T^{r+1}\mid x_0\cdot\ldots\cdot x_r=a\}\to\P^r\setminus\{x_0\cdot\ldots\cdot x_r=0\}$ 
described in~\eqref{eq-id3}
the fiber $F=f^{-1}(a)=\{\mathbf{x}\in\T^{r+1}\mid x_0\cdot\ldots\cdot x_r=a, x_i\neq\a_{i,j}\}$ is sent to
\begin{equation*}
\{[z_0:z_1:\dots:z_r]\in\P^r\mid z_i\neq 0, \, a z_0\neq \a_{0,j}z_1, \, z_i\neq \a_{i,j}z_{i+1}, \, z_r\neq \a_{r,j}z_{0}\}.\qedhere
\end{equation*}
\end{proof}

\begin{rem} By Proposition~\ref{prop-strat}, we have that $m=\prod_{i=0}^r (n_i-1)\geq m':=\#\L_f$.
Let us denote
$$\C_{m'+1}^*=\C^*\setminus \L_f^*, \quad E=\prod_{i=0}^{r}\C^*_{n_i}\setminus f^{-1}(\L_f^*).$$
The restriction $f=f_{\mathbf{n}}: E \rightarrow \C_{m'+1}^*$
is a locally trivial fibration. 
For the generic values of Proposition~\ref{prop-strat} we have $m=m'$.
For non-generic values of the $\a_{i,j}$'s the cardinal of the bifurcation set could very well be 
smaller than~$m$.

For example the map $f=f_{3,\dots, 3}: (\C^*\setminus{\{\pm1\}})^{r+1} \rightarrow \C^*$
will have $\L_f=\{\pm 1\}$, thus $m'=2<2^{r+1}=m$.
More generally, one may take the set of $\{\a_{i,j}\}_{1\le j<n_i}$ to consist of
the subgroup $\mu_{n_i-1}$ of roots of unity of order $n_i-1$ in $\C^*$.  
\end{rem}

We assume from now that $m=m'$.
The exact sequence in homotopy of the fibration $f=f_{\mathbf{n}}: E \rightarrow \C_{m+1}^*$
gives a short exact sequence of groups
\begin{equation}\label{eq-shexsq} 
1\rightarrow \pi_1(F) \rightarrow \pi_1(E) \rightarrow \pi_1(\C_{m+1}^*)\rightarrow 1,
\end{equation}
where $F$ is the fiber of $f$.
Identify $\pi_1(\C_{m+1}^*)$ with the free group of rank $m+1$ written
as $\Z*\F_m$, where $\Z$ is generated by the class $t$ of a meridian
around the origin, and $\F_m$ is the free group on $a_{\mathbf{j}}$
(coming from meridians around the points in $\L_f$), $\mathbf{j}\in C$.

Note that $\pi_1(E)$ surjects onto the Artin group $A=\prod_{i=0}^{r} \F_{n_i}$,
and the epimorphism is induced by the inclusion 
$E\hookrightarrow \prod_{i=0}^{r}\C^*_{n_i-1}$.
The inclusion $\C^*\setminus \L_f\hookrightarrow  \C^*$ induces an epimorphism
$\Z*\F_m\to\Z$ that satisfies $t\mapsto 1$ and $a_{\mathbf{j}}\mapsto 0$. 
Comparing the short exact sequence of the fiber bundle
with the one defining the Bestvina-Brady group, we obtain a commutative diagram 
whose vertical arrows are epimorphisms. In fact we have the following.

\begin{lem}\label{lema-iso-fiber}
The fundamental group of the arrangement defined in~\eqref{eq:arrangement}
is isomorphic to the Bestvina-Brady group~$N=N_{\mathbf{n}}$ for almost every~$a\in\mathbb{C}^*$.
\end{lem}

\begin{rem}
The particular case $f=f_{2,\dots, 2}: (\C^*\setminus{\{1\}})^{r+1} \rightarrow \C^*$
leads to interesting hyperplane arrangements. 
The fiber $F=f^{-1}(a)$ over $a\neq 1$ is homeomorphic to the complement
of the hyperplane arrangement $\A_{a}$ in $\P^{r}$ defined by 
$z_0 z_1\dots z_r (az_0-z_1)(z_1-z_2)\dots (z_{r-1}-z_r)(z_r-z_0)$. 
The bifurcation set $\L_f$ is just $\l=1$.
If $r>2$ then it is readily seen that the fundamental group of
the complement to $\A_{a}$ does not change even when $a=1$.
Lemma~\ref{lema-iso-fiber} implies that $N_{2,\dots, 2}$ is the group of the graphic arrangement 
defined by the polynomial
$$
z_0 z_1 \dots z_r (z_0-z_1)(z_1-z_2)\dots (z_{r-1}-z_r)(z_r-z_0), \quad r> 2.
$$
\end{rem}

\begin{proof}[Proof of Lemma{\rm~\ref{lema-iso-fiber}}]
We have the following commutative diagram, where the rows are exact
and the vertical arrows are surjective:
\begin{equation*}
\begin{matrix}
0&\to& \pi_1(F)&\to&\pi_1(E)  &\to&\pi_1(\C_{m+1}^*)=\Z*\F_m&\to&0\\
 &   &\downarrow&  &\downarrow&   &\downarrow   &            & \\  
0&\to&N        &\to&A         &\to&\pi_1(\C^*)=\Z       &\to&0
\end{matrix}
\end{equation*}
Let us fix $\mathbf{j}:=(j_0,\dots,j_r)\in C$. Let $\g_{\mathbf{j}}$ be a lift of $a_{\mathbf{j}}$ to $\pi_1(E)$
that becomes trivial in the Artin group $A$. Let $G$ be the subgroup
of $\pi_1(E)$ normally generated by $\g_{\mathbf{j}}$, $\mathbf{j}\in C$.
If we identify $\pi_1(F)$ as a subgroup of $\pi_1(E)$, then
$K:=\ker(\pi_1(F)\to N)=G\cap\pi_1(F)$. It can be easily checked 
that $K$ is normally generated by $[g,\g_{\mathbf{j}}]$, $g\in\pi_1(F)$, $\mathbf{j}\in C$.
We need to prove  that $K$ is trivial, i.e.,  
$g^{\g_{\mathbf{j}}}=g$ for all $g$ in $\pi_1(F)$ and $\mathbf{j}\in C$.
Let us prove this claim.

The crucial observation is that the conjugation by $\g_{\mathbf{j}}$ is obtained
by means of the monodromy action of $a_{\mathbf{j}}$ on $\pi_1(F)$; this is due to the choice
of the lift $\g_{\mathbf{j}}$ to be trivial in~$A$. 
In order to understand the action, we need to understand the behavior of~$f$
near the bifurcation value~$\alpha_{\mathbf{j}}$. The only stratum
not transversal to $f^{-1}(\alpha_{\mathbf{j}})$ is the stratum~$I_{\mathbf{j}}$
which consists of a point $p_{\mathbf{j}}:=(\a_{0,j_0},\dots,\a_{r,j_r})$.
Let $\mathbb{D}_{\mathbf{j}}$ be a small closed disk around $\alpha_{\mathbf{j}}$
and let us denote $\mathbb{D}_{\mathbf{j}}^*:=\mathbb{D}_{\mathbf{j}}\setminus\{\alpha_{\mathbf{j}}\}$.
The disk is chosen in order to ensure the following facts:
\begin{itemize}
\item The monomial map~$p$ defined in~\eqref{eq-p} has a  \emph{small} Milnor ball 
$\B_{\mathbf{j}}\subset p^{-1}(\mathbb{D}_{\mathbf{j}})$ for the point~$p_{\mathbf{j}}$.
\item $X_{\mathbf{j}}:=\overline{\f^{-1}(\mathbb{D}_{\mathbf{j}})\setminus\B_{\mathbf{j}}}$ 
fibers trivially over $\mathbb{D}_{\mathbf{j}}$.
\end{itemize}
We have realized a decomposition
$f^{-1}(\mathbb{D}_{\mathbf{j}}^*)=X_{\mathbf{j}}^*\cup\check{\B}_{\mathbf{j}}$ for
$X_{\mathbf{j}}^*:=X_{\mathbf{j}}\cap E$ and
$\check{\B}_{\mathbf{j}}:=\B_{\mathbf{j}}\cap E$. The intersections
$Y_{\mathbf{j}}:=X_{\mathbf{j}}\cap\check{\B}_{\mathbf{j}}$
$Y_{\mathbf{j}}^*:=X_{\mathbf{j}}^*\cap\check{\B}_{\mathbf{j}}$ are arcwise connected.

Since the statement deals with equalities of type $g^{\g_{\mathbf{j}}}=g$,
we may choose as~$F$ any non-special fiber.
Let us fix $\tilde{\alpha}_{\mathbf{j}}\in\partial\mathbb{D}_{\mathbf{j}}$
and fix $F:=f^{-1}(\tilde{\alpha}_{\mathbf{j}})$; for later use, let us denote
$F_{\mathbf{j}}:=f^{-1}(\alpha_{\mathbf{j}})$, which is a special fiber.
Let us denote $M:=F\cap\check{\B}_{\mathbf{j}}$, $M_{\mathbf{j}}:=F\cap\check{\B}_{\mathbf{j}}$,
$\check{F}:=F\cap X_{\mathbf{j}}$, $\check{F}_{\mathbf{j}}:=F_\mathbf{j}\cap X_{\mathbf{j}}$,
$F^{\partial}=M\cap\check{F}$ and $F_{\mathbf{j}}^{\partial}=M_{\mathbf{j}}\cap\check{F}_{\mathbf{j}}$.

By the above construction: 
$$
X_{\mathbf{j}}\cong\check{F}\times\mathbb{D}_{\mathbf{j}}
\cong\check{F}_{\mathbf{j}}\times\mathbb{D}_{\mathbf{j}},
\quad
X_{\mathbf{j}}^*\cong\check{F}\times\mathbb{D}_{\mathbf{j}}^*
$$
and these homeomorphisms restrict to:
$$
Y_{\mathbf{j}}\cong F^{\partial}\times\mathbb{D}_{\mathbf{j}}*
\cong F_{\mathbf{j}}^{\partial}\times\mathbb{D}_{\mathbf{j}},\quad
Y_{\mathbf{j}}^*\cong F^{\partial}\times\mathbb{D}_{\mathbf{j}}^*;
$$
moreover the pairs $(\check{F},F^{\partial})$ and $(\check{F}_{\mathbf{j}},F_{\mathbf{j}}^{\partial})$ 
are isotopic in $X_{\mathbf{j}}$. We can choose a representative of $\gamma_{\mathbf{j}}$
in $Y_{\mathbf{j}}^*$; the above homeomorphism imply that the action
of $\gamma_{\mathbf{j}}$ by conjugation on $\pi_1(\check{F})$ is trivial.
Hence, in view of the Seifert-van Kampen theorem, it is enough to check that the action
of $\gamma_{\mathbf{j}}$ by conjugation on $\pi_1(M)$ is also trivial.

This is done as follows. Using the Taylor expansion of $p$ around $p_{\mathbf{j}}$
we can choose analytic coordinates $(y_0,\dots,y_r)$ for $\B_{\mathbf{j}}$ 
(centered at $p_{\mathbf{j}}$) and for $\mathbb{D}_{\mathbf{j}}$
(centered at $\a_{\mathbf{j}}$ after a translation) such that $p(y_0,\dots,y_r)=y_0+\dots+y_r$. 

Then, $M$ has the homotopy type  of the complement of the arrangement
defined by the equation $z_1\cdot\ldots\cdot z_r\cdot(z_1+\dots+z_r-b)=0$ in $\C^r$, for some $b\in\C^*$.
The group $\pi_1(M)$ is generated by the meridians of the hyperplanes and it is abelian (since $r\geq 2$).
Let $\mu$ be the meridian of one of the hyperplanes; the element $\mu^{\gamma_{\mathbf{j}}}$
is constructed geometrically using the monodromy around $a_{\mathbf{j}}$ and it is again a meridian
of the same hyperplane. Since the group is abelian we obtain that $\mu^{\gamma_{\mathbf{j}}}=\mu$;
since the group $\pi_1(M)$ is generated by these meridians we obtain that $\gamma_{\mathbf{j}}$
acts trivially on $\pi_1(M)$ as required.
\end{proof}

\begin{proof}[Proof of Theorem{\rm~\ref{thm-111}}]
It is an immediate consequence of Lemmas~\ref{lema-hyperplane} and~\ref{lema-iso-fiber}.
\end{proof}

\begin{proof}[Proof of Corollary{\rm~\ref{cor-111}}]
In order to obtain the arrangement $\A_{\mathbf{n}}$ in $\P^2$
we only need to take a generic $2$-dimensional
slice of the hyperplane arrangement $\A_{a}$ in $\P^{r}$.
\end{proof}

We now look at the general case of $f=f_{\mathbf{n}}^{\mathbf{d}}$, where 
$d_i\ge 1$ are integers whose greatest common divisor is $1$.
The fibers of the product map $p=p^{\mathbf{d}}:\T^{r+1}\to\C^*$
are still tori isomorphic to $\T^r$, but  the identification map does not have an explicit
form in terms of the $d_i$'s. In fact, the torus $\T^r$ 
naturally lives as $\P^r(\mathbf{d})\setminus\{z_0\cdots z_r=0\}$ inside the weighted
projective space $\P^r(\mathbf{d})$. Therefore the fibers of $f$ can be identified
with complements of hypersurfaces in $\P^r(\mathbf{d})$.

\begin{rem}
In case one of the $d_i$'s is equal to $1$, say $d_0=1$, we can still realize the 
fibers of $f$ as complements of hypersurfaces in $\P^r$. Indeed, the fiber 
$p^{-1}(a)=\{x_0\cdot x_1^{d_1}\cdot\ldots\cdot x_r^{d_r}=a\}\subset\T^{r+1}$ is homeomorphic to 
$\P^r\setminus\{z_0\cdot\ldots\cdot z_r=0\}$ via the mapping: $x_i=z_i z_{i+1}^{-1}$
for $1\le i\le r$, and $x_0=az_0^{d_r}z_1^{-d_1}z_2^{d_1-d_2}\cdots z_r^{d_{r-1}-d_r}$.
\end{rem}

\begin{example}
Consider the map $f=f_{2,\dots,2}^{\mathbf{d}}: (\C^*\setminus{\{1\}})^{r+1} \rightarrow \C^*$
and assume $d_0=1$. The bifurcation set is $\L_f=\{1\}$. The fiber $F=f^{-1}(a)$ is homeomorphic 
to the complement of the arrangement of hypersurfaces $\A_{a}$ in $\P^r$ defined by the polynomial 
\[
z_0 \cdots z_r \cdot \prod_{i=1}^r (z_i-z_{i+1}) \cdot 
\prod_{j=1}^{n_0-1}(a z_0^{d_r} - z_1^{d_1}z_2^{d_2-d_1}\cdots z_r^{d_r-d_{r-1}}).
\]
\end{example}

\begin{proof}[Proof of Theorem{\rm~\ref{thm-di}}]
The proof follows the same ideas as the proof of Theorem~\ref{thm-111}. We have seen above that
the generic fibers of $f$ are in a natural way hypersurfaces of some weighted projective space.
The special fibers are obtained as in Proposition~\ref{prop-strat}; they correspond to 
some $0$-dimensional non-zero strata. If we take out the special fibers we obtain 
a short exact sequence like~\eqref{eq-shexsq}. The local behavior around the \emph{bad}
$0$-dimensional strata is like in the proof of Lemma~\ref{lema-iso-fiber}
and the result follows.
\end{proof}

\section{Fibrations and homotopy groups}\label{sec:fht}
 
We will show here that the map $f:X=\T^{r+1}\setminus D\to\C^*$, as the restriction of the
monomial map $p:\T^{r+1}\to \C^*$ given in~\eqref{eq-p}, is in fact \emph{like a fibration} in codimension $1$. 
Also, in order to investigate the homology of the Artin kernels $N_\G^\chi$, we construct 
geometric approximations of their $K(\pi,1)$ complexes.

Recall that $f:f^{-1}(B)=X_{\L}\to\C^*\setminus\L=B$ is a fiber bundle. The bifurcation set $\L$ 
consists of the elements $a\in\C^*$ for which the fiber $p^{-1}(a)$ is not transverse to all the strata
of the divisor $D$. Also recall from the proof of Proposition~\ref{prop-strat} that these
non-transversal strata are zero-dimensional and, since the stratification is finite, their union is a finite set.

\subsection{Quasifibrations}
\mbox{}

In order to describe the objects and concepts mentioned above we need some definitions. 
Let $f:X\to C$ be a surjective map where $C$ is path-connected. For $c\in C$,
let us denote $F_c:=f^{-1}(c)$.

\begin{definition}[\cite{MR0097062}]
A map $f: X\to C$ is a \emph{quasifibration} (resp. an \emph{$n$-quasifibration}) 
if it induces a weak homotopy equivalence (resp. $n$-equivalence) 
$f:(X,F_c)\to (C,c)$ for all $c\in C$, 
that is, if for any $x\in F_c$ the morphism
\[
f_*:\pi_i(X,F_c,x)\to\pi_i(C,c)
\]
is an isomorphism for all $i\ge 0$ (resp. for all $i<n$ 
and a surjection for $i=n$).
\end{definition}

A characterization of $n$-quasifibrations is the following (cf.~\cite{MR1070579,MR1745508}): 
A map $f: X\to C$ is an $n$-quasifibration if and only if the inclusion $F_c\to H_c$ 
of any fiber into the homotopy theoretic fiber $H_c$ is a (weak) homotopy $n$-equivalence. 

As a consequence of this, all the fibers of an $n$-quasifibration have the same weak homotopy 
type up to dimension~$n$.

\subsection{Monomial quasifibrations}
\mbox{}

We will show that the monomial maps $f:X=\T^{r+1}\setminus D\to\C^*$ considered in the previous 
section, $D$ as in \eqref{eq-D}, have a quasifibration structure. For the rest of the section we assume $r\ge 3$
unless otherwise stated. 

\begin{prop}
The map $f:X=\T^{r+1}\setminus D\to\C^*$ defined by $f(\mathbf{x})=x_0^{d_0}\cdot\ldots\cdot x_r^{d_r}$ 
is an $r$-quasifibration.  
\end{prop}

\begin{proof}
The proof will use the tools for quasifibrations  devised in the proof of the Dold-Thom theorem 
as outlined in Hatcher~\cite[Chapter 4, Section 4.K]{hat}. The restriction $f:X_{\L}\to\C^*\setminus\L=B$ 
is a fiber bundle, and thus a fibration and a quasifibration. Choose neighborhoods $B_j$ of the points 
$q_j$ in the finite set $\L$ so that $B\cup\bigcup_j B_j$ is a cover of $\C^*$. The result will follow from 
the gluing theorem in~\cite[Lemma 4K.3]{hat} (see also \cite[Theorem 2.3]{MR1745508}) as soon as we prove 
that the restrictions $f:f^{-1}(B_j)\to B_j$ are $r$-quasifibrations. 

Recall that to each $q_j\in\L$ it corresponds a non-transversality point
$p_j$ in the intersection of $D$ with the fiber $f^{-1}(\l)$ of $f:\T^{r+1}\to\C^*$.
Using \cite[Corollary 4K.2]{hat}, 
it is enough to prove the following local statement: 
For any $q_j\in\L$, there exist neighborhoods $U_j$ of $p_j$ in $\T^{r+1}$ and $V_j$ 
of $q_j$ in $\C^*$ such that $f:X_j=U_j\cap X\to V_j$ is an $r$-quasifibration. 

As in Lemma~\ref{lema-iso-fiber}, in local coordinates centered at $p_j$, respectively $q_j$,
we have to consider the 
map $f:\T^{r+1}\to\C$ given by $f(z)=d_0 z_0+\dots+d_r z_r$.
The fibers of this map $F_c=f^{-1}(c)$ are complements to hyperplane arrangements:
\[
F_c=\C^{r}\setminus\{z_1\dots z_{r}(d_1 z_1+\dots+ d_{r}z_{r}-c)=0\}.
\]

The arrangements involved are in general position as in the work of Hattori~\cite{hat:75}.
The results there give that $F_c$ has the following homotopy type
\[
F_c\simeq 
\begin{cases}
(T^r)^{(r)} & \text{ if } c\neq 0, \\
(T^{r-1})^{(r-1)} & \text{ if } c=0,
\end{cases}
\]
where $T^k$ is the real $k$-dimensional torus, 
and $(T^k)^{(l)}$ is its $l$-skeleton. 
That yields $\pi_1(F_c)=\Z^{r}$ and 
\[
\pi_i(F_c)=0 \text{ if } 
\begin{cases}
1< i< r & \text{ and } c\neq 0, \\
1< i< r-1 & \text{ and } c=0.
\end{cases}  
\]
The long exact homotopy sequence of the pair $(\T^{r+1},F_c)$ provides the vanishing and 
surjection needed to obtain that the map $f:(\T^{r+1},F_c)\to(\C,c)$ is an $r$-equivalence. 
Hence $f:\T^{r+1}\to\C$ is an $r$-quasifibration, and the local statement is proved.
\end{proof}

\begin{cor}\label{cor:pii}
Let $F$ be any fiber of the map $f:X=\T^{r+1}\setminus D\to\C^*$. 
Then $\pi_i(F)=0$ for $1<i< r-1$.
\end{cor}

\begin{proof}
By the previous Proposition, $f$ is an $r$-quasifibration, with aspherical base and
aspherical total space. The claim follows from the long
exact homotopy sequence of $f$.  
\end{proof}

\begin{prop}\label{prop:pin2}
Let $F$ be the general fiber of the map 
$f:X=\T^{r+1}\setminus D\to\C^*$, $r\geq 2$. Then $\pi_{r-1}(F)=0$.
\end{prop}

\begin{proof}
Note that $f:E\to B$ is a fiber bundle with fiber $F$ and aspherical base
$B=\C^*\setminus\L$. Thus we are done if we show that $\pi_{r-1}(E)=0$.

Suppose that the map $f:X=\T^{r+1}\setminus D\to\C^*$ defined by 
$f(\mathbf{x})=x_0^{d_0}\cdot\ldots\cdot x_r^{d_r}$ has $m$ special fibers $\L=\{a_1,\dots,a_m\}$
(see Proposition~\ref{prop-strat}). In this case,
$E=X\setminus\bigcup_{j=1}^m\{x_0^{d_0}\cdot\ldots\cdot x_r^{d_r}=a_j\}$.

Consider now the divisor $\overline{D}=D\cup\bigcup_{j=1}^m \{x_{r+1}-a a_j^{-1}=0\}$
in $\T^{r+2}$, for some $a\in\C^*$. It is readily seen that $E$ can be 
identified with the special fiber $\overline{F}_a=\bar{f}^{-1}(a)$ of the map 
$\bar{f}:\overline{X}=\T^{r+2}\setminus\overline{D}\to\C^*$ given by 
$\bar{f}(\mathbf{x},x_{r+1})=x_0^{d_0}\cdot\ldots\cdot x_r^{d_r} x_{r+1}$. 
On the other hand, $\pi_{r-1}(\overline{F}_a)=0$ by Corollary~\ref{cor:pii}, and so we are done.
\end{proof}

Let $F$ be the general fiber of the map 
$f:X=\T^{r+1}\setminus D\to\C^*$ defined by $f(\mathbf{x})=x_0^{d_0}\cdot\ldots\cdot x_r^{d_r}$,
where $D=\bigcup_{i=0}^r\bigcup_{j=1}^{n_i-1}\{x_i=\a_{i,j}\}$. 
Then $X$ is clearly homotopy equivalent to $T_{\G}\subset T^{r+1}$ the 
toric complex associated with the multipartite graph $\G=K_{\mathbf{n}}$. 
Let $T_{\G}^{\chi}$ be the infinite cyclic cover of $T_{\G}$ associated
with the epimorphism $\chi:A_{\G}\to\Z$ defined by $\chi(\s_{i,j})=d_i$.
The multiplication map $f:T_{\G}\to\mathbb{S}^1$, now seen in the
real context, is a fibration whose fiber has the homotopy type of $T_{\G}^{\chi}$.  
Recall that $T_{\G}^{\chi}$ is an Eilenberg-MacLane space for the
Artin kernel $N_{\G}^{\chi}=\ker\chi$.

\begin{cor}
For $r\geq 2$, the general fiber $F$ is a smooth quasi-projective variety
of dimension~$r$ that has the homotopy type of an $r$-dimensional 
CW-complex. The $r$-skeleton of the infinite cyclic cover 
$T_{\G}^{\chi}$ has the homotopy type of $F$. In particular 
\[
H_i(N_{\G}^{\chi})\cong
\begin{cases}
H_i(F)& \text{ for } i\leq r,\\
\pi_{r}^*(F), & \text{ for } i= r+1,
\end{cases} 
\]  
where $\pi_i^*$ are the coinvariants of $\pi_i$ under the $\pi_1$-action.
\end{cor}

\begin{proof}
We know that $\pi_1(F)=N_{\G}^{\chi}$. From Corollary~\ref{cor:pii} and 
Proposition~\ref{prop:pin2} we have that $\pi_i(F)=0$ for $1<i<r$. 
That ensures that $(T_{\G}^{\chi})^{(r)}$ and $F$ are homotopy equivalent.   
\end{proof}

We obtain that the first non-vanishing homotopy group of the general fiber 
$\pi_{r}(F)$ is of infinite rank, as $\pi_{r}^*(F)$ is so.    

All the special fibers $F_s$ are smooth $r$-dimensional quasi-projective varieties
sharing the homotopy type of an $r$-dimensional 
CW-complex. If $r\ge 3$, the $(r-1)$-skeletons of $F_s$ and of the general fiber 
$F$ have the same homotopy type.  

\begin{cor}
The special fiber $F_s$ is a smooth $r$-dimensional quasi-projective variety. 
If $r=2$, the special fiber $F_s$ is aspherical and $\pi_1(F_s)$ is not
isomorphic to $N_{\G}^{\chi}$. If $r\ge 3$, we have that 
\[
\pi_i(F_s)
\begin{cases}
=N_{\G}^{\chi},&\text{ if }i=1,\\
=0,&\text{ if } 1<i<r-1,\\
\neq 0&\text{ if } i=r-1,
\end{cases}
\quad
H_i(N_{\G}^{\chi})\cong H_i(F_s) \text{ for } i<r
\]
and an exact sequence 
\begin{equation}
\label{eq-long-exact}
H_{r}(F_s)\to H_{r}(N_{\G}^{\chi})\to\pi_{r-1}^*(F_s)\to 0.
\end{equation}
In other words, if $r\geq 3$, then 
$\pi_{r-1}^*(F_s)$ is a finitely generated abelian group of rank $b_{r}(F_s)-b_{r}(F)$.
\end{cor}

\begin{proof}
We know that $\pi_1(F_s)=N_{\G}^{\chi}$, if $r\ge 3$. From Corollary~\ref{cor:pii} 
we have that $\pi_i(F_s)=0$ for $1<i<r$. 
That ensures $F_s^{(r-1)}=F^{(r-1)}=(T_{\G}^{\chi})^{(r-1)}$ up to homotopy.
The exact sequence~\eqref{eq-long-exact} comes from the long exact sequence in 
homology obtained from the classifying map $F_s\to T_{\G}^{\chi}=K(N_{\G}^{\chi},1)$.
\end{proof}

% \bibliographystyle{amsplain}
% \bibliography{biblio_ea}
\def\cprime{$'$}
\providecommand{\bysame}{\leavevmode\hbox to3em{\hrulefill}\thinspace}
\providecommand{\MR}{\relax\ifhmode\unskip\space\fi MR }
\providecommand{\MRhref}[2]{%
  \href{http://www.ams.org/mathscinet-getitem?mr=#1}{#2}
}
\providecommand{\href}[2]{#2}

\end{document}